\documentclass[10pt,a4paper]{amsart}
\usepackage[utf8x]{inputenc}
\usepackage{amsmath, amsthm, bm}
\usepackage{amsfonts}
\usepackage{amssymb}
\usepackage{graphicx}
\usepackage{parskip}
\usepackage{latexsym,amsthm,amsmath,amssymb,amsfonts}
\usepackage[all]{xy}
\usepackage{color}
\usepackage{enumerate}

\newtheorem{theorem}{Theorem}[section]
\newtheorem{proposition}[theorem]{Proposition}

\newtheorem{lemma}[theorem]{Lemma}
\newtheorem{corollary}[theorem]{Corollary}

\theoremstyle{definition}
\newtheorem{remark}[theorem]{Remark}

\newtheorem{question}[theorem]{Question}

\newcommand{\lexge}{ \underset{\mathit{Lex}}{\succeq}}

\DeclareMathOperator{\HF}{HF}
\DeclareMathOperator{\HS}{HS}
\DeclareMathOperator{\ann}{Ann}
\DeclareMathOperator{\rank}{rank}

\title{On generic principal ideals in the Exterior algebra}

\author{Samuel Lundqvist and Lisa Nicklasson}
\begin{document}

\begin{abstract}
We give a lower bound on the Hilbert series of the exterior algebra modulo a principal ideal generated by a generic form of odd degree and disprove a conjecture by Moreno-Soc\'ias and Snellman. We also show that the lower bound is equal to the minimal Hilbert series in some specific cases.
\end{abstract}

\maketitle
\section{Introduction}
Given a graded algebra and an ideal generated by a generic form in this algebra, what is the Hilbert series of the quotient? This question is central in the study of the Lefschetz properties and the Fr\"oberg conjecture \cite{frobergconj} for graded commutative algebras, and much attention is drawn to this area. See the overview papers \cite{fr-lu,mi-na} for references.

We consider the question above for the exterior algebra on $n$ generators, continuing the pioneering work by Moreno-Soc\'ias and Snellman \cite{mo-sn}. 

Let $V$ be an $n$-dimensional vector space over $\mathbb{C}$, spanned by $x_1, \ldots, x_n$. We use the notation $E = \bigwedge V$, and $E_k= \bigwedge^k V$, so that $E= \bigoplus_{k=0}^n E_k$. 
 A form of degree $d$ in $E$ can be considered as a point in $A=\mathbb{A}_\mathbb{C}^N$, where $N={\binom{n}{d}}$, and it is well known that generic forms belong to a Zariski-open dense subset of $A$ on which the Hilbert series for the corresponding algebras is constant and minimal in the coefficient-wise sense.

For $d$ even, Moreno-Soc\'ias and Snellman
 proved that the minimal Hilbert series is equal to $[(1-t^d)(1+t)^n]$. Let us recall the argument. Let $f$ be a form in $E$ of degree $d$. The Hilbert series of  $E$ is equal to $(1+t)^n$, 
and if the multiplication map $\cdot f: E_i \to E_{i+d}$ is either surjective or injective for all $i$, then it is an easy exercise to show that the Hilbert series of $E/(f)$ is equal to $[(1-t^d)(1+t)^n]$, and that this series is the lowest possible. Here $[\ldots]$ means ''truncate before the first non-positive term''. By using a combinatorial argument, they then showed that $E/(h_d)$, where $h_d$ denotes the sum of all monomials of degree $d$, has series equal to $[(1-t^d)(1+t)^n]$ when $d$ is even. Thus $[(1-t^d)(1+t)^n]$ is the minimal Hilbert series.

The aim of this paper is to study the situation when $f$ has odd degree, which turns out to be surprisingly hard.
A first analysis of the odd situation was done in \cite{mo-sn}, and based upon extensive computer calculation, two conjectures for the generic series were given. 
The first conjecture, \cite[Conjecture 6.1]{mo-sn},  concerns odd $d \geq 5$, while the second conjecture,  \cite[Conjecture 6.2]{mo-sn}, treats the case $d = 3$.

Our main result is Corollary \ref{cor:mainthm_HF}, which gives a lower bound for the generic series, disproving their first conjecture. The lower bound is based upon two observations that were not part of the analysis in loc. cit.; that the Hilbert series of $E/\ann(f)$ is symmetric about $\frac{n-d}{2}$, and a rank deficiency result for the linear map $\cdot f: E_{(n-d)/2} \to E_{(n+d)/2}$ when $n$ is odd.  However, our results supports the second conjecture, and we show in Proposition \ref{prop:MSSconj} that  \cite[Conjecture 6.2]{mo-sn} agrees with our lower bound (for $d = 3$).
We are also able to show that the lower bound  is equal to the minimal Hilbert series in some special cases, by finding explicit elements in the corresponding open sets. 

Our main shortcoming is that we have not been able to show what the minimal series is, and that we have not been able to give a general construction of an odd form that attains our lower bound, that is, we do not have an equivalent to $h_d$ from the even case.  
  (The form $h_d$ has a linear factor when $d$ is odd, so the corresponding Hilbert series cannot be minimal.)

To give a flavor of the complex situation that arises, consider the following four algebras; $E(7,3), E(7,5), E(9,3), E(11,5)$, where $E(n,d)$ denotes the exterior algebra on $n$ generators modulo a principal ideal generated by a generic form of degree $d$. We can prove that they posses the following Hilbert series, denote by $\HS(\cdotp,t)$;
\begin{align*}
\HS(E(7,3),t) &= [(1+t)^7 (1-t^3)],\\
\HS(E(7,5),t) &= [(1+t)^7 (1-t^5)]  + t^6,\\
\HS(E(9,3),t) &= [(1+t)^9 (1-t^3)]  + 4 t^6,\\
\HS(E(11,5),t) &= [(1+t)^{11}(1-t^5)]  + t^8. 
\end{align*}
From the general lower bound, it follows that $\HS(E(7,3),t) \geq [(1+t)^7 (1-t^3)]$, so to prove equality, it is enough to find one cubic form such that the series is attained, and this is done in Proposition \ref{prop:smalld} by using Macaulay2 \cite{M2}.

The extra term in the series for $E(7,5)$ means that the map $ \cdot f$ from degree one to degree six in the exterior algebra on seven generators has a one-dimensional kernel, and follows as a special case from Proposition \ref{prop:n-2}. 

In the third situation, the form $f$ itself is in the kernel of the map $\cdot f$ from degree three to degree six in the exterior algebra on nine generators, so we expect a difference from 
$[(1+t)^9 (1-t^3)]$ by $t^6$. But in fact, we show in Proposition \ref{prop:n9d3} that the kernel has dimension four, and this deficiency is explained by using the classification of trivectors given in  \cite{Vinberg}. 

Finally, that the kernel of the map from degree three to degree eight when $n = 11$ has dimension at least one is explained in Proposition \ref{prop:dual_rank} as a consequence of the well known result that a skew-symmetric matrix has even rank. Equality is settled by a calculation in Macaulay2.

The paper is organized as follows. In Section \ref{sec:duality} we derive duality results, which in particular will give us the rank deficiency criteria. In Section \ref{sec:proofs}, we will present our lower bound.  In Section \ref{sec:sharp_bounds}, we will show that the lower bound is equal to the minimal series in some special cases. Section \ref{sec:conj} concerns the conjectures by Moreno-Soc\'ias and Snellman. Finally, in Section \ref{sec:disc} we present questions and open problems related to the lower bound, which we hope will inspire the reader to continue the investigation. The central question of this paper is Question \ref{question:main}  --- Is our lower bound equal to the minimal series, except when $(n,d) \neq (9,3)$?

We will use the short notation $x_ix_j$ for $x_i \wedge x_j$. Notice that, in the exterior algebra, homogeneous left and right ideals coincide. Throughout this paper, all ideals will be homogeneous, and hence there is no need to distinguish between left and right ideals.

\section{Duality} \label{sec:duality}

Given a form $f \in E_d$, let $\cdot f : E_i \to E_{i+d}$ denote the linear map defined by $a \mapsto fa$. In this section we will see that the maps $\cdot f : E_i \to E_{i+d}$ and $\cdot f : E_{n-i-d} \to E_{n-i}$ are closely related. In a particular choice of basis, this will also give additional information about the map $\cdot f : E_{(n-d)/2} \to E_{(n+d)/2}$, in the case $n-d$ is even. There is no restriction on the degree $d$ in this section. In Section \ref{sec:proofs} we will apply the results, in the case $d$ is odd.  

Let us introduce some notation. For a set $I=\{i_1, \ldots, i_m\}$ of $m$ positive integers $i_1<i_2< \dots < i_m\le n$, let $x_I:=x_{i_1} \cdots x_{i_m}$. Let $\hat{x}_I=x_{j_1} \cdots x_{j_{n-m}}$ be the product given by $\{j_1, \ldots, j_{n-m}\} = \{1, \ldots, n\} \setminus I$ and $j_1 < \dots < j_{n-m}$.  Hence $\hat{x}_Ix_I=\pm x_1 \cdots x_n$.  It follows that $x_Ix_J\ne 0$ if and only if $I \cap J = \emptyset$. For a set of monomials $x_{I_1}, \ldots, x_{I_s}$, with non-zero product, and $s\ge2$, let $\sigma(I_1, \ldots, I_s)= \pm 1$ be defined by 
 \[ \sigma(I_1, \ldots, I_s)x_{I_1} \cdots x_{I_s} = x_{j_1} x_{j_2} \cdots x_{j_t}, \ \mbox{where} \ j_1<j_2< \dots < j_t.\]
 In other words, $\sigma(I_1, \ldots, I_s)$ is the sign obtained when we rewrite $x_{I_1} \cdots x_{I_s}$ so that the indices of the variables are in increasing order. Also, let $\sigma(I)$ be defined by 
 \[\sigma(I)\hat{x}_Ix_I=x_1 \cdots x_n. \] 
 
The following two properties can easily be verified. 

\begin{itemize}
 \item For $I,J$ and $K$ such that $x_Ix_Jx_K \ne 0$,  
 \[\sigma(I,J,K) = (-1)^{|I||J|}\sigma(J,I,K) = (-1)^{|J||K|}\sigma(I,K,J).\] 
 \item If, in addition, $I \cup J \cup K = \{1, \ldots, n\}$, then 
 \[ \sigma(I,J)\sigma(K)=\sigma(I,J,K). \]
\end{itemize}

Let $\mathbb{B}_m$ and $\widehat{\mathbb{B}}_m$ denote the two bases for $E_m$ given by all monomials $x_J$ of degree $m$, and all monomials $\sigma(I)\hat{x}_I$ of degree $m$, respectively. 

\begin{lemma}\label{lemma:dualmap}
 Let $f \in E_d$, and let $M$ be the matrix of $\cdot f : E_m \to E_{m+d}$ w.r.t. the bases $\mathbb{B}_m$ and  $\widehat{\mathbb{B}}_{m+d}$. Then the matrix of $\cdot f : E_{n-m-d} \to E_{n-m}$ w.r.t. the bases $\mathbb{B}_{n-m-d}$ and  $\widehat{\mathbb{B}}_{n-m}$ is given by $(-1)^{m(n-m-d)}M^t$. 
\end{lemma}

Before we prove Lemma \ref{lemma:dualmap}, we will make a small comment about the definition of the matrix $M$, and its transpose. Formally, the matrix of a linear map depends on the ordering of the basis elements. However, we can describe the matrix $M$ as the rows being indexed by all $I=\{i_1, \ldots, i_{n-m-d}\}$, and the columns by all $J=\{j_1, \ldots, j_m\}$. In the same way, the rows in the matrix of $\cdot f : E_{n-m-d} \to E_{n-m}$ are indexed all $J=\{j_1, \ldots, j_m\}$, and the columns by all $I=\{i_1, \ldots, i_{n-m-d}\}$. The matrix $M^t$ refers to the matrix whose element at position $(J,I)$ is the same as the element at position $(I,J)$ in the matrix $M$. 

\begin{proof}
 The element at position $(I,J)$ in the matrix $M$ is the coefficient of $\sigma(I)\hat{x}_I$ in $fx_J$. This is $0$ if $I \cap J \ne \emptyset$. If $I$ and $J$ are disjoint, let $K=\{k_1, \ldots, k_d\}$ be given by $ K = \{ 1, \ldots, n\} \setminus ( I \cup J).$ Let $\alpha_K$ be the coefficient of $x_K$ in $f$. Then the coefficient of $\hat{x}_I$ in $fx_J$ is $\sigma(K,J)\alpha_K$. It follows that the element at position $(I,J)$ in the matrix M is 
 \[ \sigma(I)\sigma(K,J)\alpha_K = \sigma(K,J,I) \alpha_K. \]
 
Let $M'$ denote the matrix of $\cdot f : E_{n-m-d} \to E_{n-m}$. In the same way as above we obtain the element at position $(J,I)$ in $M'$ as $0$ if $I \cap J \ne \emptyset$, and 
 \[ \sigma(K,I,J) \alpha_K = (-1)^{|I||J|} \sigma(K,J,I) \alpha_K =(-1)^{m(n-m-d)} \sigma(K,J,I) \alpha_K, \] 
 where $K$ is the same as above, if $I \cap J = \emptyset$. It follows that $M'=(-1)^{m(n-m-d)}M^t$. 
\end{proof}

From Lemma \ref{lemma:dualmap} we can now draw the following conclusions. 

\begin{proposition}\label{prop:dual_rank}
 Let $f \in E_d$. Then the maps $\cdot f : E_m \to E_{m+d}$ and $\cdot f : E_{n-m-d} \to E_{n-m}$ have the same rank. 
 
 Moreover, suppose $n-d$ is an even number. When $(n-d)/2$ is even, the map $f:E_{(n-d)/2} \to E_{(n+d)/2}$ can be represented by a symmetric matrix. When $(n-d)/2$ is odd it can be represented by a skew-symmetric matrix, and hence has even rank. 
\end{proposition}

\begin{proof}
 The first part follows directly from Lemma \ref{lemma:dualmap}, since a matrix always has the same rank as its transpose. 
 
 Suppose that $n-d$ is even, and let $m=(n-d)/2$ in Lemma \ref{lemma:dualmap}. In this case $m=n-m-d$, so the two matrices are the same. It follows that the matrix $M$ of $f:E_{(n-d)/2} \to E_{(n+d)/2}$, w. r. t. the bases $\mathbb{B}_{(n-d)/2}$ and $\widehat{\mathbb{B}}_{(n+d)/2}$,  satisfies
 \[ M=(-1)^{((n-d)/2)^2}M^t. \]
 Hence it is symmetric when $(n-d)/2$ is even, and skew-symmetric when $(n-d)/2$ is odd. 
\end{proof}

\begin{figure}
 \includegraphics{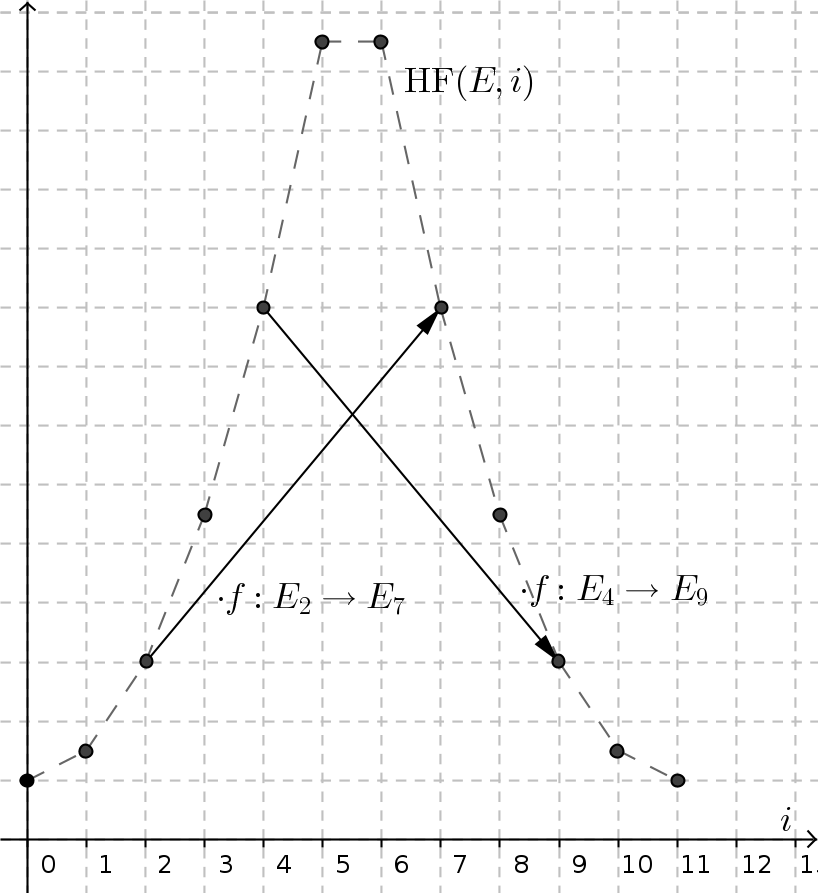}
 \caption{Here $E$ denotes the exterior algebra on 11 variables. The two maps have the same rank, according to Proposition \ref{prop:dual_rank}. }
\end{figure}

\begin{remark}\label{rmk:dibag}
Dibag \cite{dibag} showed that  if $n$ is odd, every element in $E_{n-2}$ has a linear factor. Using Proposition \ref{prop:dual_rank}, we obtain an alternative proof of this fact: Suppose $n$ is odd, and $f \in E_{n-2}$. By Proposition \ref{prop:dual_rank} the map $\cdot f : E_{1} \to E_{n-1}$ has even rank. Since $n$ is odd, it must have a non-trivial kernel. That is to say, there is an $\ell \in E_1$ such that $\ell f=0$. After a change of basis we can assume that $\ell=x_1$. But it is clear that $x_1f=0$ implies that $f = x_1 f'$ for some $f'$, concluding the argument.
\end{remark}

Proposition \ref{prop:dual_rank} also has important implications on the Hilbert function, here denoted by $\HF$. 

\begin{proposition}\label{prop:hilb_E/f}
 Let $f \in E_d$. Then 
 \[
   \HF(E/(f), i) = \binom{n}{i}-\binom{n}{i-d} + \HF(E/(f),n+d-i),
 \]
 for all $d \le i \le n$. It follows that the Hilbert function $\HF(E/(f),i)$ is completely determined by its values for $i \le \frac{n+d}{2}$. Moreover, when $\frac{n-d}{2}$ is an odd integer, we have 
 \[
 \HF\big(E/(f), \tfrac{n+d}{2}\big) \equiv \binom{n}{(n+d)/2} \mod 2.
\]
\end{proposition}
\begin{proof}
Clearly $\HF(E/(f),i)=\dim E_i = \binom{n}{i}$ for $i<d$, and $\HF(E/(f),i)=0$ for $i>n$. For $d \le i \le n$,
 \[
  \HF(E/(f), i) = \dim(E_i)-\rank(\cdot f: E_{i-d} \to E_{i}).
 \]
Proposition \ref{prop:dual_rank} gives
\begin{align*}
 \HF(E/(f), i) &= \dim(E_i) - \rank(\cdot f: E_{n-i} \to E_{n-i+d}) \\
 &=\dim(E_i)-\dim(E_{n-i+d})\\
 & \quad +\dim(E_{n-i+d}) - \rank(\cdot f: E_{n-i} \to E_{n-i+d})\\
 &= \binom{n}{i}-\binom{n}{n-i+d} + \HF(E/(f),n-i+d)\\
 &= \binom{n}{i}-\binom{n}{i-d} + \HF(E/(f),n+d-i) .
\end{align*}
Now suppose that $\frac{n-d}{2}$ is an odd integer. We have 
\[
 \HF\big(E/(f), \tfrac{n+d}{2}\big) = \dim(E_{(n+d)/2})-\rank(\cdot f: E_{(n-d)/2} \to E_{(n+d)/2}),
\]
and by Proposition \ref{prop:dual_rank}, the rank in the right hand side above is even. This proves the second part of Proposition \ref{prop:hilb_E/f}.
\end{proof}

\begin{remark}
  By Proposition \ref{prop:hilb_E/f}, the Hilbert function of $E/(f)$ is completely determined by the ranks of the maps $\cdot f: E_{i-d} \to E_{i}$ for $i \le \frac{n+d}{2}$. Notice that this is exactly the maps for which $\dim(E_{i-d}) \le \dim(E_i)$. See also Figure \ref{fig:maps}.
\end{remark}

\begin{figure}\label{fig:maps}
 \includegraphics{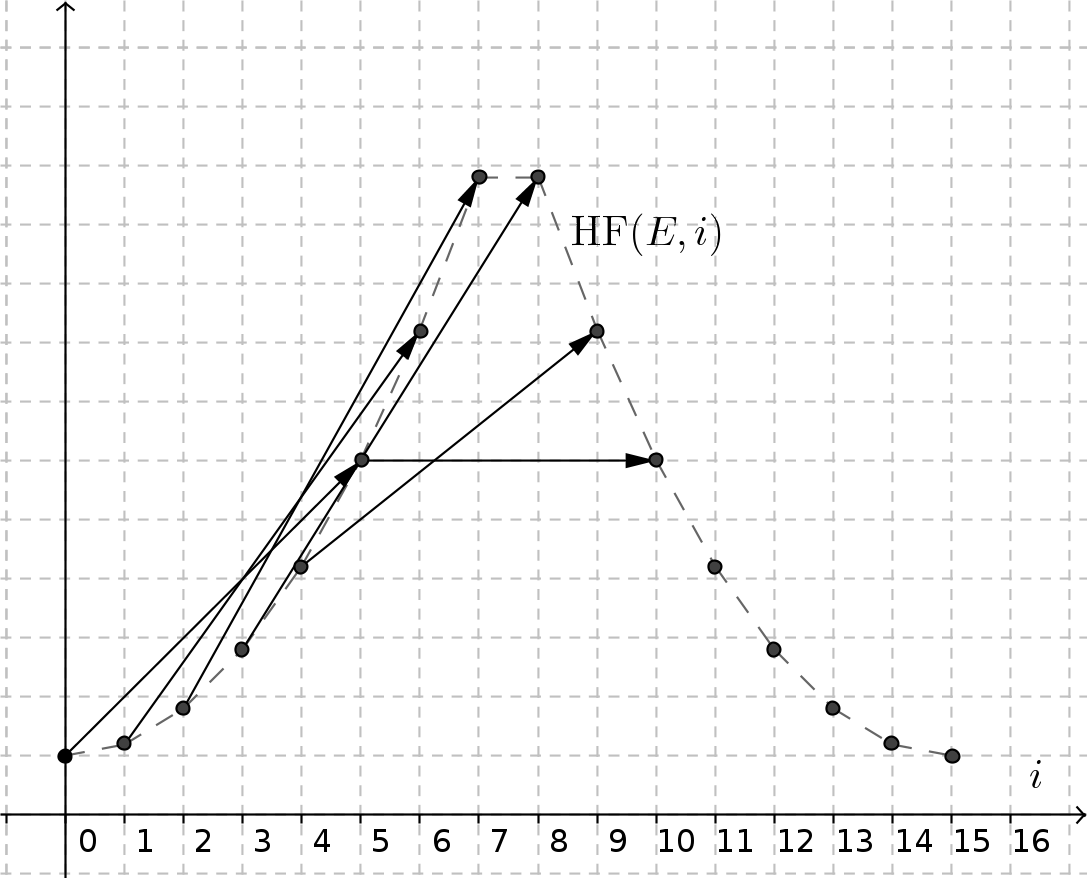}
 \caption{Here $E$ denotes the exterior algebra on 15 variables. For $f \in E_5$, the Hilbert function of $E/(f)$ is completely determined by the ranks of the maps $\cdot f : E_{i-5} \to E_i$, for $i=5,. . . , 10$. }
\end{figure}

\begin{remark}\label{rmk:hilb_E/ann}
 For a form $f$ of degree $d$, the Hilbert function of $E/\ann(f)$ is given by 
 \[
  \HF(E/\ann(f), i) = \rank(\cdot f : E_i \to E_{i+d}). 
 \]
 Notice that $ \HF(E/\ann(f), i) =0$ if $i>n-d$. For $i \le n-d$ we get 
 \[\HF(E/\ann(f), i) = \HF(E/\ann(f), n-i-d),\]
 by Proposition \ref{prop:dual_rank}. It follows that the Hilbert function of $E/\ann(f)$ is symmetric about $i=\frac{n-d}{2}$. 
\end{remark}

\section{A lower bound for the Hilbert series of $E/(f)$} \label{sec:proofs}

We will now use the results from Section \ref{sec:duality} to give a lower bound for the Hilbert series $E/(f)$, in the case $f$ is a form of odd degree. Recall that $f^2=0$, when $f$ is a form of odd degree, and hence $\ann(f)_i \supseteq (f)_i$. 

\begin{theorem}\label{thm:main}
Let $f$ be a form of odd degree $d$. Suppose $\ann(f)_i = (f)_i$, for all $i \le r \le \frac{n-d}{2}$. Then
\[ \HF(E/(f),i) =\sum_{\substack{k \ge 0 \\ kd \le i}}(-1)^k \binom{n}{i-kd}, \ \mbox{for all} \ i \le r+d.\]
Suppose next that $\ann(f)_i = (f)_i$, for all $i < r \le \frac{n-d}{2}$, and that $\ann(f)_r \supset (f)_r$. Then 
\[  \HF(E/(f),r+d) =\sum_{\substack{k \ge 0 \\ kd \le r+d}}(-1)^k \binom{n}{r+d-kd} + \dim(\ann(f)_r/(f)_r).   \]
\end{theorem}

\begin{proof}
Suppose first that $\ann(f)_i=(f)_i$, for all $i \le r \le \frac{n-d}{2}$.  Then 
\begin{equation} \label{eq:complex}
 0 \to \dots \overset{\cdot f}{\to} E_{i-d} \overset{\cdot f}{\to} E_{i} \overset{\cdot f}{\to} E_{i+d} \to [E/(f)]_{i+d} \to 0 
\end{equation}
is an exact sequence, for all $i\le r$. It follows that 
\begin{equation}\label{eq:HF_lower}
\HF(E/(f),i)=\dim[E/(f)]_i =\sum_{\substack{k \ge 0 \\ kd \le i}} (-1)^k \dim E_{i-kd} = \sum_{\substack{k \ge 0 \\ kd \le i}}(-1)^k \binom{n}{i-kd},
\end{equation}
for all $i \le r+d$. Suppose next that $\ann(f)_i = (f)_i$, for all $i < r \le \frac{n-d}{2}$, and that $\ann(f)_r \supset (f)_r$. Then (\ref{eq:HF_lower}) holds for all $i<r+d$. We get
\begin{align*}
 \HF(E/(f),r+d) &= \dim(E_{r+d}) -\rank(\cdot f : E_{r} \to E_{r+d}) \\
 &=\binom{n}{r+d} - \dim(E_r)+ \dim(\ker(\cdot f: E_r \to E_{r+d})) \\
 &= \binom{n}{r+d} - \dim(E_r)+\rank(\cdot f : E_{r-d} \to E_r)\\
 &\quad +\dim(\ker(\cdot f: E_r \to E_{r+d})) -\rank(\cdot f : E_{r-d} \to E_r)\\
 &= \binom{n}{r+d} - \HF(E/(f),r)  + \dim(\ann(f)_r)-\dim((f)_r)\\
 &= \sum_{\substack{k \ge 0 \\ kd \le m+d}}(-1)^k \binom{n}{m+d-kd}+ \dim(\ann(f)_r/(f)_r). \ \ \ \qedhere
\end{align*}
\end{proof}

\begin{corollary}\label{cor:mainthm_HF}
For a form $f\in E$ of odd degree $d$, let $h_i=\HF(E/(f),i)$, and let $\bm{h}$ be the sequence $(h_0, h_1, \ldots, h_n)$. Then $\bm{h} \lexge \bm{a},$
where $\bm{a} = \bm{b} + \bm{c},$ with 

$$
b_i = 
\begin{cases}
\displaystyle \sum_{\substack{k \ge 0 \\ kd \le i}}(-1)^k \binom{n}{i-kd} 		&\mbox{if }  0 \le i \leq \frac{n+d}{2} \\
\displaystyle \sum_{\substack{k \ge 1 \\ kd+i \le n}}(-1)^{k+1} \binom{n}{i+kd} &\mbox{if }   \frac{n+d}{2}<i \le n,
\end{cases}
$$
and

\begin{equation} \label{eq:ci}
c_i= \begin{cases}
			1 & \mbox{if } i = \frac{n+d}{2}, \ \text{and the numbers} \ \frac{n-d}{2} \text{and } b_{(n-d)/2} \ \text{are both odd}, \\
		 	0 & \mbox{otherwise.}
	\end{cases}
\end{equation}

\end{corollary}

\begin{proof}
 We know that $\ann(f)_i \supseteq (f)_i$ for all $i$. Suppose we have $r$ such that $\ann(f)_i = (f)_i$, for all $i < r \le \frac{n-d}{2}$, and that $\ann(f)_r \ne (f)_r$. Then $h_i=a_i$ for all $i<r+d$, and $h_{r+d}>a_{r+d}$, by Theorem \ref{thm:main}. Now, suppose that $\ann(f)_i = (f)_i$, for all $i < \frac{n-d}{2}$. Then $h_i=a_i$ for all $i < \frac{n+d}{2}$, and $h_{(n+d)/2} \ge b_{(n+d)/2}$, by Theorem \ref{thm:main}. In the case $\frac{n-d}{2}$ is odd we have
 \[
  h_{(n+d)/2} \equiv \binom{n}{(n+d)/2} \mod 2,
 \]
  by Proposition \ref{prop:hilb_E/f}. If, in this case, $b_{(n+d)/2} \not \equiv \binom{n}{(n+d)/2} \mod 2$, we get that $h_{(n+d)/2} \ge b_{(n+d)/2}+1$. Notice that $b_{(n+d)/2}=\binom{n}{(n+d)/2} -b_{(n-d)/2}$. Hence the condition $b_{(n+d)/2} \not \equiv \binom{n}{(n+d)/2} \mod 2$ is equivalent to $b_{(n-d)/2}$ being odd. It follows that $h_{(n+d)/2} \ge a_{(n+d)/2}$. If $h_{(n+d)/2} = a_{(n+d)/2}$ we can compute the values of $h_i$, for all $i > \frac{n+d}{2}$, by the formula given in Proposition \ref{prop:hilb_E/f}. We get 
  \begin{align*}
 \HF(E/(f),i)&= \binom{n}{i}-\binom{n}{i-d} + \HF(E/(f),n-i+d) \\
 =&\binom{n}{i} - \binom{n}{n-i+d} + \sum_{\substack{k \ge 0 \\ kd \le n-i+d}}(-1)^k \binom{n}{n-i+d-kd}  \nonumber \\
 &= \sum_{\substack{k \ge 1 \\ kd \le n-i}}(-1)^{k+1} \binom{n}{n-i-kd} = \sum_{\substack{k \ge 1 \\ i+kd \le n}}(-1)^{k+1} \binom{n}{i+kd}
\end{align*}
  We can conclude that $\bm{h} \lexge \bm{a}$, in all cases. 
\end{proof}

We can also formulate the inequality in terms of a rational series as follows.

\begin{corollary}\label{cor:mainthm_ratseries}
Let 
 \begin{equation*}
  B(t) = \frac{t^{\lfloor(n+d)/2+1 \rfloor}p(t)+(1+t)^n}{1+t^d},
 \end{equation*}
where $p(t)$ is the polynomial of degree (at most) $d-1$, determined by 
\[
 t^{\lfloor(n+d)/2+1 \rfloor }p(t) = \sum_{\frac{n+d}{2}<i\le \frac{n+d}{2}+d} \! \Big( \! \sum_{\substack{k \in \mathbb{Z} \\ 0 \le i+kd \le n}} \!\! (-1)^{k+1}\binom{n}{i+kd} \Big) t^i.
\]
Let
\[
 C(t) = \begin{cases}
               t^{(n+d)/2} & \text{if} \ \frac{n-d}{2} \ \text{is an odd integer, and the coefficient of} \ t^{(n-d)/2} \ \text{in} \\ &  B(t) \ \text{is odd}, \\[1ex]
               0 & \text{otherwise.}
              \end{cases}
\]
Then, for any form $f \in E$ of odd degree $d$, 
\[
 \HS(E/(f),t) \lexge  B(t) + C(t).
\]
\end{corollary}
Notice also that $p(t)$ is divisible by $(1+t)$, since both $1+t^d$ and $(1+t)^n$ are so, when $d$ is odd. 
\begin{proof}
Let $b_0, \ldots, b_n$ be defined as in Corollary \ref{cor:mainthm_HF}. Notice that 
\[
 b_i+b_{i-d} = \binom{n}{i}, \ \mbox{when} \ d \le i \le \frac{n+d}{2}, \ \ \mbox{or} \ \ \frac{n+d}{2}+d <i \le n. 
\]
Also, $b_i=\binom{n}{i}$ when $i \le d$. For $(n+d)/2<i \le (n+d)/2+d$ we get
\begin{align*}
  b_i+b_{i-d} &=\sum_{\substack{k \ge 1 \\ kd+i \le n}}(-1)^{k+1} \binom{n}{i+kd} + \sum_{\substack{k \ge 1 \\ i-kd \ge 0}}(-1)^{k+1} \binom{n}{i-kd}\\
  &= \sum_{\substack{k \in \mathbb{Z} \\ 0 \le i+kd \le n}}\!\!\!\!\!\! (-1)^{k+1}\binom{n}{i+kd} \ + \binom{n}{i}.
\end{align*}
Then
\begin{align*}
 \Big( \sum_{i=0}^n b_i t^i \Big) (1+t^d) &= \sum_{i=0}^{d-1} b_i + \sum_{i=d}^n(b_i+b_{i-d})t^i \\
 &= \sum_{\frac{n+d}{2}<i\le \frac{n+d}{2}+d} \! \Big( \! \sum_{\substack{k \in \mathbb{Z} \\ 0 \le i+kd \le n}} \!\! (-1)^{k+1}\binom{n}{i+kd} \Big) t^i  + \sum_{i=0}^n \binom{n}{i}t^i \\
 &= t^{\lfloor(n+d)/2+1 \rfloor}p(t)+(1+t)^n,
\end{align*}
and hence $B(t)=\sum_{i=0}^n b_i t^i$. It follows from Corollary \ref{cor:mainthm_HF} that 
\[\HS(E/(f),t) \lexge  B(t) + C(t). \qedhere \]
\end{proof}
See Lemma \ref{lemma:series_multisect_poly} for an explicit expression for $p(t)$, in the case $d=3$. 

\begin{remark} \label{remark:ann}
Notice that we have equality in Corollary \ref{cor:mainthm_HF} and Corollary \ref{cor:mainthm_ratseries} if and only if 
\[
 \dim(\ann(f)_i)=\dim((f)_i)+c_i, \ \mbox{for} \ i \le \frac{n-d}{2}, \ \mbox{where}
\]
\[
 c_i=\begin{cases}
    1 & \text{if} \ i=\frac{n-d}{2}, i \ \text{is odd, and} \ \dim((f)_i) \not \equiv \binom{n}{i} \mod 2 \\
    0 & \text{otherwise.} 
   \end{cases}
\]
\end{remark}

\section{The minimal Hilbert series of $E/(f)$}\label{sec:sharp_bounds}
When $d=n$ or $d=n-1$ the bound in Corollary \ref{cor:mainthm_HF} is equal to the minimal Hilbert series, which is trivial since $\frac{n-d}{2}<1$ in both cases. We will now see that the bound is equal to the minimal Hilbert series also in the cases $d=n-2$ and $d=n-3$. 

\begin{proposition}\label{prop:n-2}
Let $n$ be odd and let $f$ be the form
\[
 f=\sum_{1\le i_1 <i_2 < \dots < i_{n-2} \le n}\!\!\!\!\!\!\!\!x_{i_1}x_{i_2} \cdots x_{i_{n-2}}.
\]
 of degree $n-2$ in $E$. Then 
$$HS(E/(f),t)= 1 + nt + \binom{n}{2}t^2 + \cdots + \binom{n}{n-3}t^{n-3} + \left(\binom{n}{n-2} -1\right)t^{n-2} + t^{n-1},$$
each coefficient agreeing with the lower bound in Corollary \ref{cor:mainthm_HF}.
\end{proposition}

\begin{proof}
Since $n$ is odd, there is an element $\ell \in E_1$ such that $\ell f =0$, as we saw in Remark \ref{rmk:dibag}. The bound in Corollary \ref{cor:mainthm_HF} is sharp if the degree one annihilators of $f$ are all multiples of the same element $\ell$.  This is equivalent to the the multiplication map $\cdot f :E_1 \to E_{n-1}$ having rank $n-1$. 

Let us choose the basis $\{x_1, \ldots, x_n\}$ for $E_1$, and $\{ \hat{x}_1, \ldots, \hat{x}_n\}$ for $E_{n-1}$. 
The matrix of $\cdot f :E_{1} \to E_{n-1}$ w. r. t. these bases is
\[
 \begin{pmatrix}
  0 	& 1 	& -1 	& 1	& \cdots & -1 \\
  1	& 0 	& -1	& 1	& \cdots & -1 \\
  1	& -1 	& 0	& 1	& \cdots & -1 \\
  1	 & -1	& 1	& 0	& 	& \vdots \\
  \vdots & 	& 	&  & \ddots	& -1 \\
  1 & 	& 	&  & 	& 0 
 \end{pmatrix}
\]
Subtracting row $i+1$ from row $i$, for $i=1, \ldots n-1$, gives the matrix
\[
 \begin{pmatrix}
 -1 	& 1 	& 0	&0	& \cdots & 0 \\
 0 	& 1 	& -1	& 0	&\cdots	& 0 \\
 0 	& 0 	& -1	& 1	&	& \vdots \\
 \vdots &	 & 	& \ddots & \ddots &  0\\
	0&	 & 	& 	 & 1 & -1 \\
	1&\cdots & 	& \cdots & 1 & 0 \\
 
 \end{pmatrix}
\]
which has rank $n-1$.

Using the notation from Corollary \ref{cor:mainthm_HF}, we have
$$b_i =
\displaystyle \sum_{\substack{k \ge 0 \\ k (n-2) \le i}}(-1)^k \binom{n}{i-k(n-2)}  = 
\begin{cases}
	 \binom{n}{i}	&\mbox{if }  i \leq n-3\\
	 \binom{n}{i}- \binom{n}{i-(n-2)}	= \binom{n}{n-2} - 1 &\mbox{if }  i=n-2\\
	  \binom{n}{i}- \binom{n}{i-(n-2)}	= 0 & \mbox{if }  i=n-1,
\end{cases}
$$
and
$b_i = \displaystyle \sum_{\substack{k \ge 1 \\ k(n-2)+i \le n}}(-1)^{k+1} \binom{n}{i+k(n-2)} =0$ when $\frac{n+(n-2)}{2}=n-1<i \le n.$

We also have $c_{n-1} =  1$ since both $\frac{(n-(n-2))}{2} = 1$ and $b_{\frac{n-(n-2)}{2}} = n$ are odd, so the series coincides with the lower bound in Corollary \ref{cor:mainthm_HF}.

 \end{proof}
 
\begin{proposition}\label{prop:n-3}
Let $n$ be even and let $f$ be the form 
\begin{align*}
 f = &x_1  \cdots x_{n−3} + x_2  \cdots x_{n−2} + x_3 \cdots x_{n−1}+  x_4 \cdots x_{n} +\\
 & \quad + x_5 \cdots x_{n}x_1 +\cdots+ x_{n-1}x_n x_1 \cdots x_{n−5} + x_n x_1 x_2 \cdots x_{n−4}
\end{align*}

of degree $n-3$ in $E$. Then 
\begin{align*}  &HS(E/(f),t)= \\
&1 + nt + \binom{n}{2}t^2 + \cdots + \binom{n}{n-4}t^{n-4} + \! \left(\!\!\binom{n}{n-3}-1 \!\right)\!t^{n-3} + \!\left(\!\! \binom{n}{n-2} -n\! \right)\!t^{n-2}, \end{align*}
each coefficient agreeing with the lower bound in Corollary \ref{cor:mainthm_HF}.
\end{proposition}
\begin{proof}
By Remark \ref{remark:ann}, we need to prove that $\ann(f)_1=\{0\}$. Notice that 
$$x_1 x_{n-1} f = \pm \hat{x}_n, \ x_2 x_n  f = \pm \hat{x}_1,\  x_3 x_1 f = \pm \hat{x}_2,  \ldots, \ x_n x_{n−2}f = \pm \hat{x}_{n−1}.$$
This proves that $\cdot f: E_2 \to E_{n-1}$ is surjective, which by Proposition \ref{prop:dual_rank} is equivalent to $\cdot f: E_1 \to E_{n-2}$ being injective.

A simple calculation shows that the series agrees with the lower bound.
\end{proof}
\begin{remark}
The next case to consider would be $d = n-4$, but we have not been able to find a polynomial giving the required series. The corresponding cyclic 
polynomial from the case $d = n-3$ gives the correct series when $n \leq 7$, but for $n \geq 8$, the form is far from giving a minimal series. Indeed, when $n = 8$, the kernel of the map
$f: E_2 \to E_6$ has dimension six.
\end{remark}

We now turn to the exceptional case $(n,d) = (9,3)$, where our lower bound does not apply.

\begin{proposition}\label{prop:n9d3}
Let $E$ be the exterior algebra on nine variables, and let $f = 2p_1 + 2p_2 + p_3 + p_4$, where
\begin{align*}
p_1  &= x_1 x_2 x_3 + x_4 x_5 x_6 + x_7 x_8 x_9, \\
p_2  &= x_1 x_4 x_7 + x_2 x_5 x_8 + x_3 x_6 x_9, \\
p_3  &= x_1 x_5 x_9 + x_2 x_6 x_7 + x_3 x_4 x_8, \\
p_4 	&= x_1 x_6 x_8 + x_2 x_4 x_9 + x_3 x_5 x_7.
\end{align*}
Then 
\begin{align*}
&HS(E/(f),t) =\\ &1 + \binom{9}{1} t + \binom{9}{2} t^2 +  \left( \binom{9}{3} -1 \right) t^3 +  \left( \binom{9}{4} - \binom{9}{1} \right) t^4  + \left( \binom{9}{5} - \binom{9}{2} \right) t^5 +  4 t^6,
\end{align*}
and this is the minimal Hilbert series of $E/(f)$, for a form $f$ of degree $3$. 
\end{proposition}
\begin{proof}
Following Vinberg-Ela\v svili \cite{Vinberg}, let
$A$ be the set of all $(\lambda_1,\lambda_2,\lambda_3,\lambda_4)\in \mathbb{C}^4$ such that 
\begin{align*}
\lambda_1 \lambda_2 \lambda_3 \lambda_4 &\neq 0 \\
(\lambda_2^3 + \lambda_3^3 + \lambda_4^3)^3 - (3\lambda_2 \lambda_3 \lambda_4)^3 &\neq 0,\\
(\lambda_1^3 + \lambda_3^3 + \lambda_4^3)^3 - (3\lambda_1 \lambda_3 \lambda_4)^3 &\neq 0,\\
(\lambda_1^3 + \lambda_2^3 + \lambda_4^3)^3 - (3\lambda_1 \lambda_2 \lambda_4)^3 &\neq 0,\\
(\lambda_1^3 + \lambda_2^3 + \lambda_3^3)^3 - (3\lambda_1 \lambda_2 \lambda_3)^3 &\neq 0.
\end{align*}
Let $\Omega$ be the set of all elements in $E_3$ that, after a linear change of variables, can be written on the form 
\[
 \lambda_1p_1+\lambda_2p_2+\lambda_3p_3+\lambda_4p_4, \ \mbox{with} \ (\lambda_1, \lambda_2, \lambda_3, \lambda_4) \in A.
\]
It is proved in \cite{Vinberg}, that $\Omega$ is a Zariski-open subset of $E_3$. By Theorem \ref{thm:A}, the set of forms of degree 3 with minimal Hilbert series is also a Zariski-open subset, so there is an element 
$f$ in the intersection of the two sets. A simple calculation shows that if $f$ is any linear combination of $p_1,p_2,p_3$ and $p_4$, then 
$$p_1 f = p_2 f = p_3 f = p_4 f = 0.$$ Hence the dimension of $E/(f)$ is at least 4 in degree $6$.

On the other hand, a calculation in Macaulay2 shows that $E_9/(2p_1 + 2p_2 + p_3 + p_4)$ has series $1 + 9t + \binom{9}{2}t^2 +  \left(\binom{9}{3} -1 \right)t^3 + \left(\binom{9}{4}  - 9 \right)t^4 + \left(\binom{9}{5} - \binom{9}{2} \right)t^5 + 4 t^6.$ This series is minimal up to degree $5$, by Theorem \ref{thm:main}, and is minimal in degree $6$ by the argument above. This proves the statement. 
\end{proof}

\begin{proposition}\label{prop:smalld}
The lower bound for the Hilbert series of $E/(f)$, where $f$ is of odd degree $d$, given in Corollary \ref{cor:mainthm_HF}, is equal to the minimal Hilbert series in the cases
\begin{itemize}
\item $d=1$,
 \item $d=3$ and $n\le 17$, $n \neq 9$
 \item $d=5$ and $n \le 13$,
 \item $d=7$ and $n \le 11$.
\end{itemize}
\end{proposition}
\begin{proof}
The case $d=1$ follows from the fact that $\ann(\ell) = (\ell)$, for every linear form $\ell$. The other cases are proved by calculations in Macaulay2. 
\end{proof}

\section{The conjectures by Moreno-Soc\'ias and Snellman} \label{sec:conj}
We now turn to the two conjectures by Moreno-Soc\'ias and Snellman. 

Recall that 
\[
 \HF(\ann(f),i) \ge \max\left(\HF((f),i), \  \binom{n}{i}-\binom{n}{i+d}\right),
\]
when $f$ is a form of odd degree $d$. This follows from the two observations, that the kernel of the map $\cdot f : E_i \to E_{i+d}$ contains $(f)_i$, and has dimension at least $\dim E_i - \dim E_{i+d}$. The first conjecture, \cite[Conjecture 6.1]{mo-sn}, states that, when $d$ is odd and $d \geq 5$ 
\[
 \HF(\ann(f),i)=c+\max\left(\HF((f),i), \  \binom{n}{i}-\binom{n}{i+d}\right),
\]
where
\begin{equation} \label{eq:c}
 c= \begin{cases}
     1 & \text{if there are integers} \ s\ge 0, \ \text{and} \  v>0, \ \text{such that} \\ 
     & \quad i=\frac{v(v+1)}{2}, \ \ n-d=\frac{1}{2}v^2+\frac{5}{2}v-1, \ \ d=5+2vs, \\[1ex]
     0 & \text{otherwise}.
    \end{cases}
\end{equation}

The conjecture was based upon computer experiments, but false, however.  The first counterexample is for $n = 21$ and 
$d = 11$. In this case, the conjecture gives $\ann(f)_i=(f)_i$ for $i \le 5$, which is to say that $\ann(f)_i=\{0\}$ for $i \le 5$.  The map $ \cdot f : E_5 \to E_{16}$
has even rank, by Proposition \ref{prop:dual_rank}. Since $\binom{21}{5}$ is odd, the map has a non-trivial kernel. 

Worth noting is that applying Proposition \ref{prop:dual_rank} gives
\begin{align*}
  \HF(\ann(f),i)= & \dim (\ker ( \cdot f: E_i \to E_{i+d})) \\
  =& \dim E_i - \rank  ( \cdot f: E_i \to E_{i+d}) \\
  =& \binom{n}{i} - \rank  ( \cdot f: E_{n-i-d} \to E_{n-i}) \\
  =& \binom{n}{i} -\left(\binom{n}{n-i-d} - \dim(\ker(  \cdot f: E_{n-i-d} \to E_{n-i})) \right)\\
  =& \binom{n}{i} -\binom{n}{i+d} + \HF(\ann(f),n-i-d),
\end{align*}
and
\[
 \HF(\ann(f),\tfrac{n-d}{2}) \equiv \binom{n}{(n-d)/2} \mod 2,
\]
when $(n-d)/2$ is an odd integer. It follows that
\[
 \HF(\ann(f),i) \ge \begin{cases}
                     \HF((f),i)+c & \text{if} \ i \le \frac{n-d}{2} \\
                     \binom{n}{i}-\binom{n}{i+d} + \HF((f),n-d-i) & \text{if} \ i > \frac{n-d}{2},
                    \end{cases}
\]
where
\[
 c= \begin{cases}
     1 & \text{if} \ \frac{n-d}{2} \ \text{is an odd integer}, i=\frac{n-d}{2}, \ \text{and} \ \HF((f),i) \not \equiv \binom{n}{i} \mod 2, \\
     0 & \text{otherwise.}
    \end{cases}
\]

The second conjecture, \cite[Conjecture 6.2]{mo-sn},  is for the case $d = 3$, and it states that the minimal Hilbert series of $E/(f)$ is given by
$$p_n(t) = \frac{t^3 L_n(t) + (1+t)^n}{1+t^3},$$
$$L_n(t) = \begin{cases} 
(3t)^{2\ell-1} (1+t)^2 & \text{ if} \ n=4\ell \\
c_1(n) t^{2\ell-1} (1+t)(1+(3^{c_2(n)} -1)t + t^2)  & \text{ if} \ n=4\ell+1 \\
(3t)^{2\ell} (1+t)^2 & \text{ if} \ n=4\ell+2 \\
(3t)^{2\ell+1} (1+t) & \text{ if} \ n=4\ell +3
\end{cases}
$$
where $c_1(n)$ and $c_2(n)$ are some positive integers.

While we refuted the first conjecture, our results support the second one.

\begin{proposition}\label{prop:MSSconj}
Conjecture 6.2 in \cite{mo-sn} is true for $(n,d) = (9,3)$ with $c_1(9) = c_2(9) = 3$, and 
$p_n(t)$ coincides with the lower bound given in Corollary \ref{cor:mainthm_ratseries} with $c_1(n) = 1, c_2(n) = \lfloor n/2 \rfloor$. It follows that the conjecture is true for $n \le 17$.
\end{proposition}
\begin{proof}
In the case $n=9$, put $c_1(9) = c_2(9) = 3$, and expand the series $p_n(t)$. This yields exactly the same series as in Proposition \ref{prop:n9d3}. 

When $n$ is even, $n-3$ is not divisible by 2, so $C(t)=0$ in Corollary \ref{cor:mainthm_ratseries}. Combining Corollary \ref{cor:mainthm_ratseries} with Lemma \ref{lemma:series_multisect_poly}, gives the lower bound 
\[ \frac{3^{m-1}t^{m+2}(1+t)^2+(1+t)^n}{1+t^3}\]
for the Hilbert series, where $n=2m$. This agrees with $p_n(t)$ when $n=4\ell$, in which case $m=2\ell$, and when $n=4\ell+2$, in which case $m=2\ell+1$. 

Assume now $n=4\ell+3$. In this case $(n-3)/2=2\ell$ is even, so $C(t)=0$ in Corollary \ref{cor:mainthm_ratseries}.  Combining Corollary \ref{cor:mainthm_ratseries} with Lemma \ref{lemma:series_multisect_poly}, gives the lower bound 
\[ \frac{3^{2\ell+1}t^{2\ell+4}(1+t)+(1+t)^n}{1+t^3},\]
which agrees with $p_n(t)$.

Last, assume $n=4\ell+1$. We get
\[B(t)=\frac{3^{2\ell}t^{2\ell+3}(1+t)+(1+t)^n}{1+t^3}\]
in Corollary \ref{cor:mainthm_ratseries}, when we use the result of Lemma \ref{lemma:series_multisect_poly}. In this case $(n-3)/2=2\ell-1$ is odd. Also $b_{2\ell -1}$ is odd, by Lemma \ref{lemma:series_multisect_odd}. Hence $C(t) = t^{2 \ell +2}$, in Corollary \ref{cor:mainthm_ratseries}. We get the lower bound
\begin{align*}
 B(t)+C(t) =& \frac{3^{2\ell}t^{2\ell+3}(1+t)+(1+t)^n}{1+t^3} + t^{2 \ell +2} \\
 =& \frac{t^{2 \ell +2}(1+t)(1+(3^{2 \ell} -1)t+t^2) +(1+t)^n}{1+t^3} 
\end{align*}
This agrees with $p_n(t)$, when $c_1(n)=1$ and $c_2(n)=2 \ell$. By Proposition \ref{prop:smalld}, the conjecture is true when $n \le 17$.
\end{proof}

\section{Discussion} \label{sec:disc}
We will now discuss questions related to the lower bound given in Corollary \ref{cor:mainthm_HF} and Corollary \ref{cor:mainthm_ratseries}.


We saw in Section \ref{sec:sharp_bounds} that the bound for the Hilbert series, given in Corollary \ref{cor:mainthm_ratseries}, is equal to the minimal Hilbert series in some special cases. One exception is the case $n=9$ and $d=3$, where we get the series from Corollary \ref{cor:mainthm_ratseries}, but with $C(t)=3t^6$ instead of $t^6$. This is the only case that we have detected, where the series does not agree with that in Corollary \ref{cor:mainthm_ratseries}. 

\begin{question} \label{question:main}
 Is the bound in Corollary \ref{cor:mainthm_ratseries} equal to the minimal Hilbert series, except when $n=9$ and $d=3$? Equivalently, when $(n,d) \ne (9,3)$, do we have
\[
 \dim(\ann(f)_i)=\dim((f)_i)+c_i, \ \mbox{for} \ i \le \frac{n-d}{2}, \mbox{where}
\]
\[
 c_i=\begin{cases}
    1 & \text{if} \ i=\frac{n-d}{2}, i \ \text{is odd, and} \ \dim((f)_i) \not \equiv \binom{n}{i} \mod 2 \\
    0 & \text{otherwise?} 
   \end{cases}
\]
\end{question}

The property $\ann(f)_i=(f)_i$, for $i < \frac{n-d}{2}$, implies that $f$ can not have a factor of odd degree less that $\frac{n-d}{2}$. Indeed, if $f=gh$ and $g$ is of odd degree, then $gf=g^2h=0$, and $g \in \ann(f)$. More generally, if $f \in (g_1, \ldots, g_s)$, where $g_1, \ldots, g_s$ are all forms of odd degree, then for any element $g\in (g_1\cdots g_s)$, we have $gf=0$. This leads to the following two questions. 

\begin{question}
 Let $f$ and $g$ be forms in $E$, such that $gf=0$, and $\deg(g) < (n-\deg(f))/2$. Does it follow that $f \in (g_1, \ldots, g_s)$, for some forms $g_1, \ldots, g_s$ of odd degree, such that $g \in (g_1 \cdots g_s)$? 
\end{question}
\begin{question}
 Are there forms $f$ in $E$, such that $f \notin (g_1, \ldots, g_s)$, for any collection of forms $g_1, \ldots, g_s$ of odd degree less than $\deg(f)$, such that $\deg(g_1 \cdots g_s) < (n-\deg(f))/2$?
\end{question}

The polynomial $p(t)$ in Corollary \ref{cor:mainthm_ratseries} can be simplified when $d=3$, see Lemma \ref{lemma:series_multisect_poly} and Proposition \ref{prop:MSSconj}. One might ask is such a simplification is possible for larger $d$. 
\begin{question}
 Is there a nice expression for the numerator of $B(t)$ in Corollary \ref{cor:mainthm_ratseries}?
\end{question}



Let $f$ be a generic form of degree $d$ in $E$, and let $\sum_{i=0}^n a_it^i$ be the Hilbert series of $E/(f)$. By Theorem \ref{thm:A}, if $\sum_{i=0}^n h_it^i$ is the Hilbert series of $E/(f')$, for any form $f'$ of degree $d$, then $h_i \ge a_i$, for each $i$. This gives rise to the following question.

\begin{question}\label{que:lexineq}
 Can we replace the lexicographic inequality by ''$h_i \ge a_i$ for all $i$'', in Corollary \ref{cor:mainthm_HF}?
\end{question}
It is natural to study this question in terms of the Hilbert function of the quotient module $\ann(f)/(f)$. Indeed,  
suppose $r$ is the first degree where $\ann(f)_r\ne (f)_r$, and suppose $r+d <\frac{n-d}{2}$. Then 
\[
 h_{r+d} =a_{r+d}+ \HF(\ann(f)/(f),r).
\]
If we continue as in the proof of Theorem \ref{thm:main} we get that 
\[
 h_{r+d+j} =a_{r+d+j}+ \HF(\ann(f)/(f),{r+j})
\]
for $j=1, \ldots d-1$. For $j=d$ we get
\[
 h_{r+2d} =a_{r+2d}+\HF(\ann(f)/(f),{r+d})-\HF(\ann(f)/(f),{r}).
\]
We see that $h_{r+2d} \ge a_{r+2d}$ if 
\[
  \HF(\ann(f)/(f),{r+d}) \ge  \HF(\ann(f)/(f),{r}).
\]
Hence, to answer Question \ref{que:lexineq}, we need to understand the Hilbert series of $\ann(f)/(f)$. It follows by Proposition \ref{prop:dual_rank} that 
\begin{align*} 
 \HF(\ann(f)/&(f),i) = \dim(\ann(f)_i) - \dim((f)_i)= \\
= & \dim(\ker(\cdot f : E_i \to E_{i+d}))-\rank(\cdot f: E_{i-d} \to E_{i}) \\
 =& \dim(E_i) - \rank (\cdot f : E_i \to E_{i+d}) -\rank(\cdot f: E_{i-d} \to E_{i}) \\
 =& \dim(E_{n-i}) - \rank (\cdot f : E_{n-i-d} \to E_{n-i}) -\rank(\cdot f: E_{n-i} \to E_{n-i+d}) \\
 =& \dim(\ker(\cdot f : E_{n-i} \to E_{n-i+d}))-\rank (\cdot f : E_{n-i-d} \to E_{n-i}) \\
 =& \HF(\ann(f)/(f),{n-i}),
\end{align*}
which means that the Hilbert function is symmetric about $i=\frac{n}{2}$. However, it is not clear whether it is weakly increasing up to $i=\frac{n}{2}$, or not.

\begin{remark}
In the case when $f$ is of even degree, the question of factorization is not as much of a problem, since it does not necessarily imply factors of odd degree. As mentioned in the introduction, the form of degree $2d$ used in \cite{mo-sn} to obtain the minimal Hilbert series of $E$ modulo a form of degree $2d$, is 
\[
 h_{2d} = \sum_{1 \le i_1 < \ldots <i_{2d} \le n} \!\!\!\!\!x_{i_1} \cdots x_{i_{2d}}.
\]
It is an easy exercise to show that $h_{2d}$ can, in fact, be factorized as 
\[
 h_{2d}=\frac{1}{d!}\Big( \sum_{1 \le i<j\le n}x_ix_j \Big)^d.
\]
It follows, for a generic form $g$ of degree 2, that $\cdot g^d : E_i \to E_{i+2d}$ has maximal rank for all $i$.  In terms of  the Lefschetz properties, this means that $E_n$ has the strong $2$-degree Lefschetz property, and that $h_2$ is a strong Lefschetz element of degree two.  See also \cite{ie-wa}, where a similar remark was made for $h_2$ being weak $2$-degree Lefschetz. 

The corresponding form 
\[
 h_{2d+1}= \sum_{1 \le i_1 < \ldots <i_{2d+1} \le n} \!\!\!\!\! x_{i_1} \cdots x_{i_{2d+1}}.
\]
is \emph{not} a good general candidate in the odd case, since it is divisible by $(x_1 + \dots + x_n)$. 
\end{remark}

{ \bf Acknowledgements.}
We would like to thank all the participants of the Stockholm problem solving seminar for many inspiring discussions during the progress of this work. We are also thankful to Jan Snellman for comments on a draft of this manuscript.

\appendix
\section{Generic forms}\label{app:gen}
The following theorem is an analog to \cite[Theorem 1]{fr-lo}, which can be applied to the exterior algebra. The theorem is folklore, but since we are not aware of a proof in the literature, we give a full proof. 

Let $R$ be a graded algebra over $\mathbb{C}$. A sequence of forms $(f_1, \ldots, f_r)$ in $R$ can be considered as a point $P$ in $\mathbb{A}_\mathbb{C}^N$, where $\deg f_i = d_i$, and $N = \sum_{i=1}^r \dim R_{d_i}$. We say that the algebra $R/I$ is induced by the point $P$ if $I$ is the left, right, or two-sided ideal generated by $f_1, \ldots, f_r$.   

\begin{theorem} \label{thm:A}
Let $R=R_0 \oplus R_1 \oplus \dots \oplus R_s$ be a finite dimensional graded algebra over $\mathbb{C}$. Let $r, d_1,\ldots,d_r$ be positive integers, and $A=\mathbb{A}_\mathbb{C}^N$, where $N = \sum_{i=1}^r \dim R_{d_i}.$ 

Then the Hilbert series of algebras induced by points in $A$ take only finitely many values. Furthermore, there is a Zariski-open dense subset of $A$ on which the Hilbert series for the corresponding algebras is constant and minimal.
\end{theorem}

\begin{proof}
Let $(f_1, \ldots, f_r)$ be a sequence of forms in $R$ corresponding to a point in $A$, and let $I$ be the left, right, or two-sided ideal generated by this sequence. The Hilbert series of $R/I$ is a polynomial $a_0+ a_1t+ \dots + a_st^s$, with $a_i \le \dim R_i$. It follows that there are only finitely many possibilities for the Hilbert series of $R/I$.

Let us first consider the case of right ideals. Let $\{p^{(j)}(t)\}_j$ be the set of polynomials $p^{(j)} = a^{(j)}_0 + a^{(j)}_1 t + \cdots +a^{(j)}_st^s$, that are Hilbert series corresponding to some point in $A$. Let $b_i = \min_j\{a^{(j)}_i\}$. Let $F_m \subset A$ be a set of points yielding Hilbert series $\sum_{i=0}^s a_i t^i$ with $a_m > b_m$. The condition $a_m >b_m$ gives a restriction on the ranks of the maps
\begin{align*}
 &\cdot f_1: R_{d_1-m} \to R_m, \\
 &\cdot f_2: [R/f_1R]_{d_2-m} \to [R/f_1R]_{m}, \\
 & \ \ \ \ \  \vdots \\
 &\cdot f_r: [R/(f_1R + \dots + f_{r-1}R)]_{d_r-m} \to  [R/(f_1R+ \dots +f_{r-1}R)]_{m},
\end{align*}
induced by multiplication by $f_i$ from the left. That a linear map has rank strictly less than a number $k$ is equivalent to all $k \times k$-minors of the matrix being zero. In this case, the minors can be expressed as polynomials in coefficients of the $f_i$'s. It follows that $F_i$ is a finite union of closed sets, which is again closed. Then, so is $\bigcup_{i=0}^s F_i$, and hence the complement of this set is a non empty open set. Thus $\sum_i b_i t^i$ actually occurs as the Hilbert series for some point in $A$, and the set of points with this series is open. 

The case when $I$ is the left ideal generated by $f_1, \ldots, f_r$ is proved analogously. Let now $I$ be the two-sided ideal generated by $f_1, \ldots, f_r$. For each $i=1, \ldots, r$, let $e_{i1}, e_{i2}, \ldots, e_{ik}$ be the monomials of degree at most $m-d_i$, and let $c_{ij}=\deg {e_{ij}}$. Let $J_1=\{0\}$, and $J_i=Rf_1R + \dots + Rf_{i-1}R$, and consider the linear maps
\begin{align*}
  \cdot e_{i1}f_i: & [R/J_i]_{m-d_i-c_{i1}} \to [R/J_i]_m, \\
  \cdot e_{i2}f_i: & [R/(J_i+e_{i1}f_iR)]_{m-d_i-c_{i2}} \to [R/(J_i+e_{i1}f_iR)]_m ,\\
 & \ \  \vdots \\
  \cdot e_{ik}f_i: & [R/(J_i+e_{i1}f_iR+ \ldots + e_{i,k-1}f_iR)]_{m-d_i-c_{is}}  \to \\
 & \qquad \ \ \ \ \ \ [R/(J_i+e_{i1}f_iR+ \ldots + e_{i,k-1}f_iR)]_m, 
\end{align*}
for $i=1, \ldots, r$. The condition $a_m>b_m$ gives a restriction on the rank of these maps. The result follows in the same way as for right ideals, since the minors corresponding to the above maps also can be expressed as polynomials in the coefficients of the $f_i$'s. 
\end{proof}

\begin{remark}
In Theorem \ref{thm:A}, the field $\mathbb{C}$ can be replaced by any infinite field.
\end{remark}

\section{An application of the Series Multisection Formula}
The purpose of this appendix is to prove Lemma \ref{lemma:series_multisect_poly} and Lemma \ref{lemma:series_multisect_odd}, which are used in the proof of Proposition \ref{prop:MSSconj}. Our main tool is the \emph{series multisection formula}. Let $F(x)=\sum_{i \ge 0} a_ix^i \in \mathbb{Z}[[x]]$. Then the $r$:th $d$-section of the series $F(x)$ is given by 
\begin{equation}\label{series_multisect_formula}
 \sum_{i \ge 0} a_{r+id}x^{r+id} = \frac{1}{d} \sum_{j=1}^d \omega^{-rj}F(\omega^jx), 
\end{equation}
where $\omega$ is a primitive $d$:th root of unity. For a proof of (\ref{series_multisect_formula}), see e.\! g.\! \cite{ri}. 

\begin{lemma}\label{lemma:series_multisect_poly}
 For a fixed positive integer $n$, 
 \begin{align*}
  &\sum_{\frac{n+3}{2}<i\le \frac{n+9}{2}} \! \Big( \! \sum_{\substack{k \in \mathbb{Z} \\ 0 \le i+3k \le n}} \!\! (-1)^{k+1}\binom{n}{i+3k} \Big) t^i =\\
  & \quad = \begin{cases}
                                                                                                                                                           3^{m-1}t^{m+2}(1+t)^2 & \text{if} \ n=2m, \ \text{for an integer} \ m, \\
                                                                                                                                                           3^mt^{m+3}(1+t) & \text{if} \ n=2m+1,\ \text{for an integer} \ m.                                                                                                                                         \end{cases}
 \end{align*}
\end{lemma}
\begin{proof}
 Let 
 \[
  F(x)=(1-x)^n = \sum_{j=0}^n(-1)^j\binom{n}{j}x^{n-j}, 
 \]
 and let $\omega$ be a primitive third root of unity. For a fixed integer $i$, let
 \[
  c_i= \sum_{\substack{k \in \mathbb{Z} \\ 0 \le i+3k \le n}} \!\!\!\!\! (-1)^{k+1}\binom{n}{i+3k}.
 \]
Write $i=3q+r$, where $q,r \in \mathbb{Z}$ and $0 \le r <3$. Then 
\begin{align*}
 c_i=\sum_{\substack{k \in \mathbb{Z} \\ 0 \le i+3k \le n}} \!\!\!\!\! &(-1)^{k+1}\binom{n}{i+3k} = -\!\!\!\!\!\!\sum_{\substack{k \in \mathbb{Z} \\ 0 \le i+3k \le n}} \!\!\!\!\! (-1)^{3k}\binom{n}{i+3k} =\\
 &= (-1)^{q+1}\!\!\!\!\!\!\!\sum_{\substack{k \in \mathbb{Z} \\ 0 \le r+3k \le n}} \!\!\!\!\! (-1)^{3k}\binom{n}{r+3k} = (-1)^{q+1+r}\!\!\!\!\!\!\!\sum_{\substack{k \in \mathbb{Z} \\ 0 \le r+3k \le n}} \!\!\!\!\! (-1)^{3k+r}\binom{n}{r+3k}.
\end{align*}
This is the $r$:th trisection of $(-1)^{q+r+1}F(x)$, evaluated at $x=1$. By (\ref{series_multisect_formula}) we get
\begin{align*}
 c_i&=(-1)^{q+r+1}\frac{1}{3}\sum_{j=1}^3 \omega^{-rj}F(\omega^j) = (-1)^{q+r+1}\frac{1}{3}(\omega^{-r}(1-\omega)^n+\omega^{-2r}(1-\omega^2)^n) \\&= (-1)^{q+r+1}\frac{1}{3}(\omega^{2r}(1-\omega)^n+\omega^{r}(1-\omega^2)^n)
\end{align*}
Now, notice that $(1-\omega)^2=1-2\omega+\omega^2=-3\omega$, since $1+\omega+\omega^2=0$. In a similar way we see that $(1-\omega^2)^2=-3\omega^2$. If $n=2m$, we get
\begin{align*}
 c_i&= (-1)^{q+r+1}\frac{1}{3}(\omega^{2r}(1-\omega)^{2m}+\omega^{r}(1-\omega^2)^{2m}) \\
 &=\! (-1)^{q+r+1}\frac{1}{3}(\omega^{2r}(-3\omega)^{m}\!+\omega^{r}(-3\omega^2)^{m}) = (-1)^{q+r+m+1}3^{m-1}(\omega^{2r+m}\!+\omega^{r+2m})
\end{align*}
If $n=2m+1$, we get
\begin{align*}
 c_i&= (-1)^{q+r+1}\frac{1}{3}(\omega^{2r}(1-\omega)^{2m+1}+\omega^{r}(1-\omega^2)^{2m+1}) \\
 &= (-1)^{q+r+1}\frac{1}{3}(\omega^{2r}(-3\omega)^{m}(1-\omega)+\omega^{r}(-3\omega^2)^{m}(1-\omega^2)) \\
 &= (-1)^{q+r+m+1}3^{m-1}(\omega^{2r+m}-\omega^{2r+m+1}+\omega^{r+2m}-\omega^{r+2m+2})
\end{align*}
It remains to compute $c_i$ for $(n+3)/2<i\le (n+9)/2$. For both $n=2m$ and $n=2m+1$, the $i$'s satisfying this inequality are $i=m+2, m+3$ and $m+4$. This gives us six cases to examine.  
\begin{enumerate}[{\bf I.}]
 \item \underline{$n=2m$ and $i=m+2$} \\
 We have $m+2=i=3q+r$, and hence $m=3q+r-2$. Then
 \begin{align*} 
 c_{m+2}=&(-1)^{q+r+m+1}3^{m-1}(\omega^{2r+m}+\omega^{r+2m})\\
 =&(-1)^{4q+2r-1}3^{m-1}(\omega^{3q+3r-2}+\omega^{6q+3r-4}) \\
 =&-3^{m-1}(\omega+\omega^2)=3^{m-1} 
 \end{align*}
\item \underline{$n=2m$ and $i=m+3$} \\
Here $m=3q+r-3$, which gives
\begin{align*} 
 c_{m+3}=&(-1)^{q+r+m+1}3^{m-1}(\omega^{2r+m}+\omega^{r+2m})\\
 =&(-1)^{4q+2r-2}3^{m-1}(\omega^{3q+3r-3}+\omega^{6q+3r-6}) \\
 =&3^{m-1}(1+1)=2\cdot 3^{m-1} .
 \end{align*}
 \item \underline{$n=2m$ and $i=m+4$} \\
Here $m=3q+r-4$, which gives
\begin{align*} 
 c_{m+4}=&(-1)^{q+r+m+1}3^{m-1}(\omega^{2r+m}+\omega^{r+2m})\\
 =&(-1)^{4q+2r-3}3^{m-1}(\omega^{3q+3r-4}+\omega^{6q+3r-8}) \\
 =&-3^{m-1}(\omega^2+\omega)=3^{m-1}. 
 \end{align*}
 \item \underline{$n=2m+1$ and $i=m+2$} \\
Here $m=3q+r-2$, which gives
\begin{align*} 
 c_{m+2}=&(-1)^{q+r+m+1}3^{m-1}(\omega^{2r+m}-\omega^{2r+m+1}+\omega^{r+2m}-\omega^{r+2m+2})\\
 =&(-1)^{4q+2r-1}3^{m-1}(\omega^{3q+3r-2}-\omega^{3q+3r-1}+\omega^{6q+3r-4}-\omega^{6q+3r-2}) \\
 =&-3^{m-1}(\omega-\omega^2+\omega^2-\omega)=0.
 \end{align*}
  \item \underline{$n=2m+1$ and $i=m+3$} \\
Here $m=3q+r-3$, which gives
\begin{align*} 
 c_{m+3}=&(-1)^{q+r+m+1}3^{m-1}(\omega^{2r+m}-\omega^{2r+m+1}+\omega^{r+2m}-\omega^{r+2m+2})\\
 =&(-1)^{4q+2r-2}3^{m-1}(\omega^{3q+3r-3}-\omega^{3q+3r-2}+\omega^{6q+3r-6}-\omega^{6q+3r-4}) \\
 =&3^{m-1}(1-\omega+1-\omega^2)=3^m .
 \end{align*}
  \item \underline{$n=2m+1$ and $i=m+4$} \\
Here $m=3q+r-4$, which gives
\begin{align*} 
 c_{m+4}=&(-1)^{q+r+m+1}3^{m-1}(\omega^{2r+m}-\omega^{2r+m+1}+\omega^{r+2m}-\omega^{r+2m+2})\\
 =&(-1)^{4q+2r-3}3^{m-1}(\omega^{3q+3r-4}-\omega^{3q+3r-3}+\omega^{6q+3r-8}-\omega^{6q+3r-6}) \\
 =&-3^{m-1}(\omega^2-1+\omega-1)=3^m .
 \end{align*}
\end{enumerate}
Finally we see that, when $n=2m$,
\begin{align*}
c_{m+2}t^{m+2}+c_{m+3}t^{m+3}+c_{m+4}t^{m+4}&=3^{m-1}t^{m+2}+2\cdot 3^{m-1}t^{m+3}+3^{m-1}t^{m+4}\\
&=3^{m-1}t^{m+2}(1+t)^2,
\end{align*}
and when $n=2m+1$, 
\begin{align*}
c_{m+2}t^{m+2}+c_{m+3}t^{m+3}+c_{m+4}t^{m+4}&=3^{m}t^{m+3}+3^{m}t^{m+4}=3^{m}t^{m+3}(1+t),
\end{align*}
which we wanted to prove.
\end{proof}

\begin{remark}
As a consequence of Lemma \ref{lemma:series_multisect_poly} we obtain some identities for alternating sums of binomial coefficients, for instance 
$$\binom{12 a}{0} - \binom{12 a}{3} + \binom{12 a}{6} - \binom{12a}{9} + \cdots + \binom{12a}{12a} = 2 \cdot 3^{6a -1}.$$

\end{remark}

\begin{lemma}\label{lemma:series_multisect_odd}
 Let $\ell$ be a positive integer. Then, the number
 \[
  \sum_{\substack{k \ge 0 \\ 3k+1 \le 2 \ell}} (-1)^k \binom{4 \ell +1}{2 \ell -1-3k} 
 \]
is odd.
\end{lemma}
\begin{proof}
 The number above is odd if and only if 
 \[
  A=\sum_{\substack{k \ge 0 \\ 3k+1 \le 2 \ell}} \binom{4 \ell +1}{2 \ell -1-3k} 
 \]
is odd. Let $r$ be the remainder of $2 \ell -1$ when divided by three. Then
\begin{align}
\sum_{\substack{k \ge 0 \\ r+3k \le 4 \ell +1}} \binom{4 \ell +1}{r+3k} & = \sum_{\substack{k \in \mathbb{Z} \\ 0 \le 2 \ell -1 +3k \le 4 \ell +1}} \binom{4 \ell +1}{2 \ell -1+3k} \label{eq:lemma_multisect_odd}\\
&= A + \sum_{\substack{k > 0 \\ 2 \ell -1 +3k \le 4 \ell +1}} \binom{4 \ell +1}{2 \ell -1+3k} \nonumber \\
&= A + \sum_{\substack{k > 0 \\ 3k \le 2 \ell +2}} \binom{4 \ell +1}{2 \ell +2-3k} \nonumber\\
&= A + \sum_{\substack{k \ge 0 \\ 3k+1 \le 2 \ell}} \binom{4 \ell +1}{2 \ell -1-3k} =2A. \nonumber
\end{align}
The left hand side of (\ref{eq:lemma_multisect_odd}) is also the $r$:th trisection of $(1+x)^{4 \ell +1}$, evaluated at $x=1$. Let $\omega$ be a primitive third root of unity. By (\ref{series_multisect_formula}) 
\begin{align*}
 2A &= \frac{1}{3} ( \omega^{-r}(1+\omega)^{4 \ell +1} + \omega^{-2r} (1+\omega^2)^{4 \ell +1} + \omega^{-3r}(1+\omega^3)) \\
 &= \frac{1}{3}( \omega^{2r}(-\omega^2)^{4 \ell +1} + \omega^r(-\omega)^{4 \ell +1} + 2^{4 \ell +1} )\\
 &= \frac{1}{3}( -\omega^{2(4\ell +r +1)}-\omega^{4 \ell +r+1} + 2^{4 \ell +1}).
\end{align*}
Since $4 \ell +r+1 \equiv 6 \ell \equiv 0  \mod 3$, we have
\[
 2A = \frac{1}{3}( -\omega^{2(4\ell +r +1)}-\omega^{4 \ell +r+1} + 2^{4 \ell +1} )= \frac{2^{4 \ell +1} -2}{3} . 
\]
Then $3A = 2^{4\ell} -1$, and we can see that $A$ is an odd number.

\end{proof}


\begin{thebibliography}{99}
\bibitem[Di]{dibag} 
I. Dibag, 
Factorization in exterior algebras, 
J. Algebra 30, 259--262 (1974).

\bibitem[Fr]{frobergconj} R. Fr\"oberg, 
An inequality for Hilbert series of graded algebras. 
Math. Scand. 56, No. 2: 117--144 (1985).

\bibitem[Fr-Lo]{fr-lo} R. Fr\"oberg and C. L\"ofwall, On Hilbert series for commutative and 
noncommutative graded algebras. J. Pure Appl. Algebra 76, 33--38 (1990).

\bibitem[Fr-Lu]{fr-lu} R. Fr\"oberg, and S. Lundqvist, 
Extremal Hilbert series. 
arXiv: 1711.01232, (2017).

\bibitem[Gr-St]{M2} D. Grayson, M. Stillman, 
Macaulay2, a software system for research in algebraic geometry, 
available at {\tt www.math.uiuc.edu/Macaulay2}.



\bibitem[Ie-Wa]{ie-wa}
S. Iyengar and M. Walker, Examples of finite free complexes of small rank and homology. arXiv: 1706.02156 (2017).

\bibitem[Mi-Na]{mi-na} J. Migliore and U. Nagel,  
Survey Article: A tour of the weak and strong Lefschetz properties. 
J. Comm. Algebra 5, No. 3, 329--358 (2013).

\bibitem[Mo-Sn]{mo-sn} 
G. Moreno-Soc\'ias, J. Snellman, 
Some conjectures about the Hilbert series of generic ideals in the exterior algebra. 
Homology, homotopy and applications 4, No. 2, 409--426, (2002).

\bibitem[Ri]{ri}
J. Riordan,
Combinatorial identities.
John Wiley \& Sons, Inc. (1968)

\bibitem[Vi-El]{Vinberg}
E.B. Vinberg, A.G. Ela\v svili,
Classification of trivectors of the 9-dimensional space.
Sel. Math. Sov. 7, No. 1 , 62--98. (1988)

\end{thebibliography}
\end{document}